\pdfoutput=1
\documentclass[reqno,oneside,letterpaper]{amsart}
\usepackage{mathrsfs}
\usepackage{amssymb}
\usepackage{wasysym}
\usepackage{bbm}

\usepackage[utf8]{inputenc}
\usepackage{amsmath, amssymb,bm, cases, mathtools, thmtools}
\usepackage{verbatim}
\usepackage{graphicx}\graphicspath{{figures/}}
\usepackage{multicol}
\usepackage{tabularx}
\usepackage[usenames,dvipsnames]{xcolor}
\usepackage{mathrsfs}
\usepackage{url}
\usepackage[normalem]{ulem}
\usepackage{xstring}
\usepackage{enumitem}

\usepackage[%
    minnames=1,maxnames=99,maxcitenames=3,
    style=alphabetic,
    url=false,
    firstinits=true,hyperref,natbib,backend=bibtex]{biblatex}
\renewbibmacro{in:}{%
  \ifentrytype{article}{}{\printtext{\bibstring{in}\intitlepunct}}}
\bibliography{biblio}

\usepackage[colorlinks,citecolor=blue,urlcolor=blue,linkcolor=RawSienna]{hyperref}
\usepackage{hypernat}
\usepackage{datetime}

\DeclareMathAlphabet\EuRoman{U}{eur}{m}{n}
\SetMathAlphabet\EuRoman{bold}{U}{eur}{b}{n}

\usepackage{euscript,microtype}

\usepackage[capitalize]{cleveref}

\crefname{assumption}{Assumption}{Assumptions}
\crefname{claim}{Claim}{Claims}

\makeatletter
\let\reftagform@=\tagform@
\def\tagform@#1{\maketag@@@{\ignorespaces\textcolor{gray}{(#1)}\unskip\@@italiccorr}}
\renewcommand{\eqref}[1]{\textup{\reftagform@{\ref{#1}}}}
\makeatother

\definecolor{WowColor}{rgb}{.75,0,.75}
\definecolor{SubtleColor}{rgb}{0,0,.50}

\newcounter{margincounter}

\declaretheorem[style=plain,numberwithin=section,name=Theorem]{theorem}
\declaretheorem[style=plain,sibling=theorem,name=Lemma]{lemma}

\declaretheorem[style=definition,sibling=theorem,name=Definition]{definition}
\declaretheorem[style=definition,name=Assumption]{assumption}

\declaretheorem[style=remark,sibling=theorem,name=Remark]{remark}

\crefformat{conditionINNER}{#2(#1)#3}
\crefmultiformat{conditionINNER}
  {(#2#1#3)}
  { and~(#2#1#3)}
  {, (#2#1#3)}
  { and~(#2#1#3)}
\Crefformat{conditionINNER}{Condition~#2(#1)#3}
\Crefmultiformat{conditionINNER}
  {Conditions~(#2#1#3)}
  { and~(#2#1#3)}
  {, (#2#1#3)}
  { and~(#2#1#3)}

\declaretheoremstyle[
    spaceabove=-6pt,
    spacebelow=6pt,
    headfont=\normalfont\bfseries,
    bodyfont = \normalfont,
    postheadspace=1em,
    qed=$\square$,
    headpunct={{}}]{myproofstyle}

\numberwithin{equation}{section}
\numberwithin{theorem}{section}

\usepackage{amssymb}
\usepackage{mathrsfs}
\usepackage{leftidx}

\usepackage{setspace}

\def\[#1\]{\begin{align}#1\end{align}}
\def\*[#1\]{\begin{align*}#1\end{align*}}

\newcommand{\Reals}{\mathbb{R}}
\newcommand{\Nats}{\mathbb{N}}

\newcommand{\PosReals}{\Reals_{> 0}}

\newcommand{\dee}{\mathrm{d}}

\DeclareMathOperator*{\newlim}{\mathrm{lim}\vphantom{\mathrm{infsup}}}
\DeclareMathOperator*{\newmin}{\mathrm{min}\vphantom{\mathrm{infsup}}}
\DeclareMathOperator*{\newmax}{\mathrm{max}\vphantom{\mathrm{infsup}}}
\DeclareMathOperator*{\newinf}{\mathrm{inf}\vphantom{\mathrm{infsup}}}
\DeclareMathOperator*{\newsup}{\mathrm{sup}\vphantom{\mathrm{infsup}}}
\renewcommand{\lim}{\newlim}
\renewcommand{\min}{\newmin}
\renewcommand{\max}{\newmax}
\renewcommand{\inf}{\newinf}
\renewcommand{\sup}{\newsup}

\newcommand{\tvd}[2]{\|#1-#2\|_{\mathrm{TV}}}

\newcommand{\BorelSets}[1]{\mathcal{B}[#1]}

\newtheorem{open problem}{Open Problem}

\newcommand{\interior}[1]{%
  {\kern0pt#1}^{\mathrm{o}}%
}

\newcommand{\refproof}[1]{See \cref{#1} for \IfSubStr{#1}{,}{proofs}{a proof}. }

\newif\iflongform
\longformtrue

\iflongform

\else

\fi

\makeatletter
\providecommand*{\toclevel@definition}{0}
\providecommand*{\toclevel@theorem}{0}
\providecommand*{\toclevel@lemma}{0}
\makeatother

\usepackage{leftidx}
\title[Drift, Minorization and Hitting Times]
{
Drift, Minorization, and Hitting Times
}

\newcommand{\expect}{\mathbb{E}}

\newcommand{\cX}{\mathcal{X}}

\begin{document}

\author[R.~M.~Anderson]{Robert M.~Anderson}
\address{University of California, Berkeley, Department of Economics}

\author[H.~Duanmu]{Haosui Duanmu}
\address{University of California, Berkeley, Department of Economics}

\author[A.~Smith]{Aaron Smith}
\address{University of Ottawa, Department of Mathematics and Statistics}

\author[J.~Yang]{Jun Yang}
\address{University of Toronto, Department of Statistical Sciences}

\maketitle

\begin{abstract}
The ``drift-and-minorization'' method, introduced and popularized in \cite{rosenthal1995minorization, meyn1994computable, meyn2012markov}, remains the most popular approach for bounding the convergence rates of Markov chains used in statistical computation. This approach requires estimates of two quantities: the rate at which a \textit{single} copy of the Markov chain ``drifts'' towards a fixed ``small set'', and a ``minorization condition'' which gives the worst-case time for \textit{two} Markov chains started within the small set to couple with moderately large probability. In this paper, we build on \cite{oliveira2012mixing,finitemixhit} and our work \cite{anderson2018mixhit,anderson2019mixhit2} to replace the ``minorization condition'' with an alternative ``hitting condition'' that is stated in terms of only \textit{one} Markov chain, and illustrate how this can be used to obtain similar bounds that can be easier to use.

\end{abstract}

\section{Introduction}

Since the seminal article by \citet{gelfand1990sampling}, Markov chain Monte Carlo methods (MCMC) have become ubiquitous in statistical computing. These methods suffer from a well-known problem: although it is easy to construct many MCMC algorithms that are guaranteed to converge \textit{eventually}, it is typically very difficult to tell \textit{how long} convergence will take for any given MCMC algorithm (see \textit{e.g.}  \cite{jones2001honest,diaconis2008gibbs}). Even worse, it is usually difficult to tell if an MCMC has converged yet, even after running it (see \textit{e.g.} \cite{gelman1992inference} for an important early paper in the ``diagnostics" literature, \cite{bhatnagar2011computational} for a proof that the problem is intractable in general, and \cite{hsu2015mixing,huber2016perfect} for cases where the problem becomes tractable). 

In the decades since the publication of \citet{gelfand1990sampling}, many techniques have been developed to  try to solve this problem for classes of MCMC algorithms that are important in statistics (see \textit{e.g.}\ the surveys by \citet{jones2001honest,diaconis2009markov}). The most popular of these techniques is the ``drift-and-minorization'' method, introduced and popularized in \citep{rosenthal1995minorization, meyn1994computable, meyn2012markov}. The purpose of this note is to introduce a similar bound in which the minorization condition (see Inequality \cref{EqDefMin}) is replaced by a hitting condition (see Definition \ref{DefEqHitCond}).  Beyond this substitution, our main result is quite similar to previous drift-and-minorization bounds. 

It is natural to ask: why would one bother to replace a minorization condition by a hitting condition? We defer a detailed discussion of this issue to  Section \ref{SecAppl}, but highlight here the main motivation: 

We \textbf{don't} make the strong claim that any chains of particular statistical interest satisfy our drift-and-hitting condition but fail to satisfy previous drift-and-minorization conditions.\footnote{We suspect that some modifications of the definitions would be needed to establish a precise equivalence for some class of Markov chains, but this is outside the scope of the current article.} In particular, we don't know of any interesting examples for which our new result can be used and it is \textit{impossible} to use previous results. 

Although we don't have examples for which a minorization condition is \textit{impossible} to establish, we give what we consider to be fairly strong evidence that it can be \textit{much harder} than a hitting condition and is never \textit{much easier}. Difficulty is of course subjective, but we think that the following provides some strong and fairly objective evidence that others will share this perception:

\begin{enumerate}
\item  We provide a specific class of examples for which a few previous papers have failed to establish a quantitatively-strong minorization bound. We provide a quantitatively-strong hitting bound, and furthermore this bound follows immediately from some well-known calculations for discrete Markov chains (and some soft arguments from analysis). Thus, there exist examples for which hitting conditions seem to be easier to obtain.
\item More generally, the question of when a hitting-time bound implies a minorization-like condition was an open question for many years, even in the setting of finite Markov chains. The finite case was eventually resolved by experts in the area in a few nontrivial papers \cite{oliveira2012mixing,finitemixhit}. For general Markov processes, similar connections were established in \cite{anderson2018mixhit} and \cite{anderson2019mixhit2}, using nonstandard techniques developed in \cite{Markovpaper}. 
    This provides some evidence that it is ``hard" to get a minorization bound from a hitting bound.
\item In the other direction, we provide a short and elementary argument showing that a minorization bound implies a quantitatively-similar hitting bound. This provides some evidence that it is ``easy" to get a hitting bound from a  minorization bound.

\end{enumerate}

Having summarized our motivation, we continue by setting some notation and recalling an important special case of the drift-and-minorization bound. Throughout this paper, unless otherwise mentioned, we use $\cX$ to denote the state space of the underlying Markov process and we always assume that $\cX$ is a topological space. Denote by $\{g(x,1,\cdot)\}_{x \in \cX}$ the transition kernel of some Markov chain on state space $\cX$ with (unique) stationary measure $\pi$ on the Borel $\sigma$-field $\BorelSets \cX$. For any $t\in \Nats$, we write $g(x,t,\cdot)$ to denote the $t$-step transition kernel generated from $\{g(x,1,\cdot)\}_{x\in \cX}$. 

We say that $g$ exhibits a ``drift'' or ``Lyapunov'' condition if there exists a function $V \, : \, \cX \mapsto [0,\infty)$ and constants $0 \leq \lambda < 1$, $0 \leq b < \infty$ so that 
\[ \label{EqDefDrift}
(gV)(x)=\int V(y) g(x,1,\dee y) \leq \lambda V(x) + b,
\]
for all $x \in \cX$. We say that $g$ exhibits a ``minorization'' condition if there exists a set $C\in \BorelSets \cX$, constants $0 < \epsilon \leq 1$, $t_0 \in \Nats$, and measure $\mu$ on $\cX$ so that 
\[ \label{EqDefMin}
g(x,t_0,\cdot) \geq \epsilon \, \mu(\cdot)
\]
for all $x \in C$. We say $g$ satisfies conditions \cref{EqDefDrift}, \cref{EqDefMin} \textit{compatibly} if it satisfies both and there exists $r > \frac{2 b}{1 - \lambda}$ such that
\[ 
C \supset \{x \in \cX \, : \, V(x) \leq r\}.
\] 

{\citep[][Thm.~12]{rosenthal1995minorization}} says that any kernel $g$ satisfying these conditions compatibly is geometrically ergodic, and also gives quantitative bounds on the convergence rate:

\begin{theorem} [Paraphrase of {\citep[][Thm.~12]{rosenthal1995minorization}}] \label{ThmDriftMinClassic}
Let $g$ satisfy conditions \cref{EqDefDrift}, \cref{EqDefMin} compatibly. Then there exists $M \, : \, \cX \mapsto [0,\infty)$ and $0 < \rho \leq 1$ such that
\[ \label{IneqDefMin}
\tvd{g(x,t,\cdot)}{\pi(\cdot)} \leq M(x) (1 - \rho)^{t}
\]
for all $x \in \cX$ and $t \in \Nats$. Moreover, there is an explicit formula for $M$ and $\rho$  in terms of the constants $\lambda,b,t_0,\epsilon,r$.
\end{theorem} 

There are more sophisticated versions of this drift-and-minorization bound, including results in the same paper \cite{rosenthal1995minorization}, but most are based on two conditions that are similar to  \cref{EqDefDrift} and \cref{EqDefMin}.

The goal of the present paper is to present results similar to Theorem \ref{ThmDriftMinClassic}, in which the minorization condition \cref{EqDefMin} has been replaced by related hitting conditions.  
Recall that the \textit{hitting time} of a set $A \in \BorelSets \cX$ for a Markov chain $\{X_{t}\}_{t \in \Nats}$ is:
\[ \label{EqBasicHitDef}
\tau(A) = \min \{ t \in \Nats \, : \, X_{t} \in A \}.
\]
Note that $\tau(A)=\infty$ if $A=\emptyset$.
We now introduce the maximum hitting time (of sets with large measure):

\begin{definition}\label{defmaxhit}
Let $\alpha\in \PosReals$. The maximum hitting time (with parameter $\alpha$) is
\[
t_{H}(g,\alpha)=\sup\{\expect_{x}[\tau(A)]: x\in X, A\in \BorelSets \cX\ \text{such that}\ \pi(A)\geq \alpha\},
\]
where $\expect_{x}$ is the expectation of a measure in the space which generates the underlying Markov process and the subscript $x$ is the starting point of the Markov process.
\end{definition}

When the kernel is clear from the context, we write $t_{H}(\alpha) = t_{H}(g,\alpha)$. We will replace Inequality \cref{EqDefMin} with an assumption about $t_{H}(g',\alpha)$ for some kernel $g'$; in practice we will take $g'$ to be an appropriate restriction of the dynamics of $g$ to some small set.

\subsection{Guide to Paper}

We give our main drift-and-hit theorem in Section \ref{SecBdDC}, show that its conditions are satisfied for various common MCMC chains in Section \ref{SecUseGMH}, and describe why it can (sometimes) give better results than the usual drift-and-minorization theorem in Section \ref{SecAppl}.

\subsection{Previous Work}

This paper is based on the asymptotic equivalence of mixing and hitting times as established in the sequence of papers  \cite{oliveira2012mixing,finitemixhit,basu2015characterization, anderson2018mixhit,anderson2019mixhit2}. This relationship allows us to show that our hitting condition (see \cref{DefEqHitCond}) implies something very close to the minorization condition \cref{EqDefMin}, and thus to use the framework of \citet{rosenthal1995minorization}. We note that our recent works \citep{anderson2018mixhit,anderson2019mixhit2} are based on the work of \citet{finitemixhit,oliveira2012mixing}, which initially established this equivalence of mixing and hitting times in the special case that $\mathcal{X}$ is finite. \citep{anderson2018mixhit} and \citep{anderson2019mixhit2} make heavy use of hyperfinite representation of Markov processes, which is established by \citet{Markovpaper}.

There is a large literature exploring the drift-and-minorization approach to bounding the convergence rate of Markov chains, including various efforts to tweak its conditions. A simple and very closely-related paper is \cite{Roberts2001}, which is also primarily concerned with replacing the minorization condition by a condition that is in some sense equivalent but sometimes easier to verify in practice. 

More generally, see \textit{e.g.} \cite{cattiaux2013poincare} and works referenced therein for a discussion of the relationship between drift-and-minorization conditions, hitting time conditions, and other ways to measure the convergence rates of Markov chains (though in a slightly different setting from the present paper). As discussed in that paper, existing equivalences are not always useful for obtaining strong quantitative bounds on convergence rates - even if two conditions give \textit{some} bound on convergence rates, they may not give \textit{similar} bounds.

There is a large literature on calculating and bounding expected hitting times. We give here a very incomplete survey of some results that may be helpful: 

It is typically straightforward to give a recurrence leading to an exact expression, though this may be difficult to use (see \textit{e.g.}  \cite[Chapter 10]{markovmix} for the discrete case). In some cases, these recurrences have been solved in a useful way (see \textit{e.g.} \cite{palacios1996note} for the simplest setting); as we see later in this paper, these hitting-time formulas are sometimes easier to use than exact mixing-time formulas, even when both are available in some form. In other situations, exact formulas for hitting times may be too complicated to use. Popular tools for estimating hitting times in these settings include spectral methods (see \textit{e.g.} the main result of \cite{miclo2010absorption}, or \cite[Lemma 8]{shapira2019time}) or using tools from the metastability literature to ``reduce" a Markov chain to simpler dynamics (see \textit{e.g.} the recent survey \cite{landim2019metastable}, which includes formulas for hitting times that may be useful in other regimes).

\section{Bounded Mixing Times for Dominated Chains} \label{SecBdDC}

We recall the definition of \emph{mixing time}:

\begin{definition}\label{defmixwp}
Fix $\epsilon\geq 0$ and a non-empty set $A\in \BorelSets \cX$. 
The \emph{mixing time} in $A$ with respect to $\epsilon$ is defined as
\[
t_{m}(\epsilon,A)=\min\left\{t\geq 0: \sup_{x\in A}\tvd{g(x,t,\cdot)}{\pi(\cdot)}\leq \epsilon\right\}.
\]
The \emph{lazy mixing time} in $A$ with respect to $\epsilon$ is
\[
t_{L}(\epsilon,A)=\min\left\{t\in \Nats: \sup_{x\in A}\parallel g_{L}(x,t,\cdot)-\pi(\cdot) \parallel_{\mathrm{TV}}\leq \epsilon\right\},
\]
where $g_{L}(x,1,\cdot)=\frac{1}{2} g(x,1,\cdot)+\frac{1}{2} \delta_{x}(\cdot)$.
\end{definition}

For notational convenience, we write $t_{m}(A)$ or $t_{L}(A)$ to mean $t_{m}(\frac{1}{4},A)$ or $t_{L}(\frac{1}{4},A)$, write $t_{m}(\epsilon)$ or $t_{L}(\epsilon)$ to mean $t_{m}(\epsilon, \cX)$ or $t_{L}(\epsilon, \cX)$, and write $t_{m}$ or $t_{L}$ to mean $t_{m}(\frac{1}{4},\cX)$ or $t_{L}(\frac{1}{4},\cX)$.

In practice, it is rare for transition kernels $g$ corresponding to MCMC algorithms on unbounded state spaces to have finite mixing times (this would require jumps of unbounded size, and these are often difficult to construct). For this reason, the mixing and maximal hitting times cannot be directly compared in a meaningful way. We will fix this problem by constructing related Markov chains that both (i) have finite mixing times and (ii) have dynamics that are ``slightly faster" than those of $g$. We introduce the following definition to formalize this idea:

\begin{definition}[Dominated Chain]
Let $g$ be a transition kernel on state space $\cX$ with associated $\sigma$-field $\BorelSets \cX$. We say that a transition kernel $g^{(S)}$ with support contained in $S$ is \textit{$S$-dominated by $g$} if it satisfies the following two properties:

\begin{enumerate}
\item The stationary measure $\pi^{(S)}$ of $g^{(S)}$ is given by
\[ \label{DefRestMeas}
\pi^{(S)}(A) = \frac{\pi(A \cap S)}{\pi(S)}
\]
for all $A \in \BorelSets \cX$.
\item The transitions of $g^{(S)}$ satisfy
\[
g^{(S)}(x,1,A) \geq g(x,1,A) 
\]
for all $x \in S$ and $A \subset S$ with $A \in \BorelSets \cX$.
\end{enumerate}

\end{definition}

There are many constructions in the literature that satisfy this $S$-domination property. The most famous of these is probably the \textit{trace} of a chain on $S$, but in this paper we focus on restrictions that are more specialized to Metropolis--Hastings and Gibbs samplers. 

\begin{definition} [Metropolis-Hastings Restriction of Sampler] \label{DefResMH}

Fix an ergodic transition kernel $g$ with stationary measure $\pi$ and fix $S \in \BorelSets \cX$ with $\pi(S) > 0$. We define the \textit{restriction} $g^{(S)}$ of $g$ to $S$ to be the transition kernel on $S$ given by the formula:
\[
g^{(S)}(x,1,A) = g(x,1,A \cap S) + \textbf{1}_{x \in S} \, (1-g(x,1,A \cap S)).
\]
\end{definition}

We note that this restricted sampler is exactly the ``usual" Metropolis-Hastings kernel with proposal $g$ and target $\pi^{(S)}$ when the latter is defined (\textit{e.g.} when $g$ is reversible with respect to some measure, and $\{g(x,1,\cdot)\}_{x \in \cX}$ and $\pi(\cdot)$ have densities with respect to a common dominating measure).

In the special case of the Gibbs sampler, another restriction is sometimes useful:

\begin{definition}[Restricted Gibbs Samplers] \label{DefRestGibbs}
Recall that the Gibbs sampler is uniquely defined by the target distribution. If $g$ is a Gibbs sampler targeting $\pi$ and $S \in \BorelSets \cX$ has $\pi(S) > 0$, we define the \textit{Gibbs restriction} $g^{(S)}$ of $g$ to $S$ to be the Gibbs sampler $g^{(S)}$ targeting the measure $\pi^{(S)}$ defined in \cref{DefRestMeas}.
\end{definition}

It is clear that both of these constructions give $S$-dominated chains, as does the usual trace process (see \textit{e.g.}  \cite[Section 6]{beltran2010tunneling} for the last result); in general, all three are different. For the remainder of this section, we will always fix a kernel $g$ and denote by $g^{(C)}$ any specific kernel that is $C$-dominated by $g$.

The following hitting condition plays an essential role throughout the paper. 

\begin{definition} \label{DefEqHitCond}
Fix $C \in \BorelSets \cX$ with $\pi(C)>0$. For $\alpha \in \PosReals$, denote by $t_{H}^{(C)}(\alpha)$ the maximum hitting time of large sets for $g^{(C)}$. We say that $g$ satisfies the \emph{$C$-hitting condition} if there exists $0 < \alpha < 0.5$ such that $t_{H}^{(C)}(\alpha) < \infty$.
\end{definition}

For a fixed collection of constants $D = \{d_{\alpha}: 0<\alpha<0.5\}$, define $\mathcal{G}(D)$ to be the collection of kernels $g$ satisfying $t_{m} \leq d_{\alpha} t_{H}(\alpha)$.

We recall that the main results in \cite{oliveira2012mixing,finitemixhit,basu2015characterization, anderson2018mixhit,anderson2019mixhit2} show that there exist \textit{universal} constants $D$ so that this is satisfied for all $g$ in certain large classes.

\begin{definition} \label{DeftDriftHitComp}
We say that $g$ satisfies the drift condition \cref{EqDefDrift} and the $C$-hitting condition (\cref{DefEqHitCond}) \textit{compatibly} for $\alpha, C, D$ if:
\begin{enumerate}
\item $g$ satisfies both the drift condition \cref{EqDefDrift} and the $C$-hitting condition.

\item $g^{(C)}$ satisfies the drift condition \cref{EqDefDrift} with the same $V,\lambda,b$ as $g$.

\item $g^{(C)} \in \mathcal{G}(D)$.

\item $C \supset \{x \in \cX \, : \, V(x) \leq r\}$ for some $r>\frac{2b}{1-\lambda}$.
\end{enumerate}
\end{definition}

It is straightforward to check that these conditions imply that $t_{H}^{(C)}(\alpha) < \infty$ for \textit{all} $0 <\alpha < 1$, though for convenience we will typically focus on only one or two values at a time.

The main result of this section, whose proof is deferred until the end, is:

\begin{theorem}\label{tracegeoergodic}

Let $g$ be a transition kernel on the state space $\cX$ that satisfies \cref{EqDefDrift} and \cref{DefEqHitCond} compatibly for some $C\in \BorelSets {\cX}$ with $\pi(C)>0$ and some $D, \alpha$. Define
\[
C' = \{x \in \cX \, : \, V(x) \leq r' \}
\]
for some $r' > \frac{2b}{1 -\lambda}$. 
Suppose $C\supset \{x\in \cX\, : \, V(x)\leq r\}$ for some $r>\frac{2b+24r'}{1-\lambda}$.

Under these assumptions, there exists $M \, : \, \cX \mapsto [0,\infty)$ and $0 < \rho \leq 1$ such that
\[
\tvd{g(x,t,\cdot)}{\pi(\cdot)} \leq M(x) (1 - \rho)^{t}
\]
for all $x \in \cX$ and $t \in \Nats$. Moreover, there is an explicit formula for $M$ and $\rho$  in terms of the constants $\lambda,b, r, D, \alpha$.
\end{theorem}

The main difference between \cref{tracegeoergodic} and \cref{ThmDriftMinClassic} is that we have replaced a \textit{minorization} condition on a small set with a \textit{hitting time} condition on a small set; the small sets themselves are of similar size, determined by the drift condition. 

Another obvious difference is that we have stated our hitting-time condition in terms of a $C$-dominated kernel $g^{(C)}$ rather than the original kernel $g$. In principle this change could have a large impact on the applicability of the results, sometimes making it harder and sometimes easier. In practice we would be surprised if this change had a large impact. In simple examples, such as the example studied in \cref{ThmSimpleApp} with sets $C$ that are intervals with stationary measure substantially over $0.5$, the three $C$-dominated chains discussed in this paper have dynamics that are quite similar to those of $g$; as a result similar analyses are possible.


We need some additional notation before proving \cref{tracegeoergodic}. Let $t_m^{(C)}$ denote the mixing time for $g^{(C)}$. 
We now show that we can obtain a minorization condition similar to \cref{EqDefMin} by combining the hitting condition in \cref{DefEqHitCond} and the drift condition in \cref{EqDefDrift}.     

\begin{theorem}\label{traceminor}
Given the assumptions of \cref{tracegeoergodic}, for $t>t_{m}^{(C)}(C')$
\[
\sup_{x,y\in C'} \| g(x,t,\cdot) - g(y,t,\cdot) \|_{\mathrm{TV}} \leq \frac{2}{3}.
\]
\end{theorem}
\begin{proof}
Fix $x\in C'$ and pick $0<\alpha<0.5$ as in \cref{DefEqHitCond}. Let $T=t_{m}^{(C)}(C')\leq d_{\alpha}t_{H}^{(C)}(\alpha)<\infty$. 
As $x\in C'$, by assumption, we have 
\[ \label{IneqTraceMinnow} 
\tvd{g^{(C)}(x,T,\cdot)}{\pi^{(C)}(\cdot)} \leq \frac{1}{4}.
\]

Next, we denote by $\{X_{t}\}_{t \in \Nats}$ a Markov chain with transition kernel $g$ starting at point $X_{0}=x\in C'$, and we denote by $\{Y_{t}\}_{t \in \Nats}$ a Markov chain with transition kernel $g^{(C)}$ starting at the same point $x$. By \cref{EqDefDrift}, we have for all $t \in \Nats$ 
\[
\mathbb{E}[V(X_{t}) | X_{t-1}] \leq \lambda V(X_{t-1})+b.
\]
 
Iterating this bound, we have 
\begin{align*}
\mathbb{E}[V(X_{t})] &\leq \lambda\mathbb{E}[V(X_{t-1})]+b \leq \lambda(\lambda\, \mathbb{E}[V(X_{t-2})] + b) + b\\
&\leq \ldots \leq \lambda^{t} V(X_{0}) + \frac{b}{1 -\lambda}.
\end{align*}

By assumption, $\max\{V(X_{0}), V(Y_{0})\} \leq r'$, so we have
\[
\max\{\mathbb{E}[V(X_{t})], \, \mathbb{E}[V(Y_{t})]\} \leq  \lambda^{t} r' + \frac{b}{1 -\lambda}
\]
for all $t \in \Nats$. By a union bound and Markov's inequality, we have
\begin{align*}
\mathbb{P}\left[\max_{0 \leq t \leq T} \{V(X_{t}), V(Y_{t})\}>r\right]
&\leq \mathbb{P}\left[\max_{0 \leq t \leq T} V(X_{t})>r\right] + \mathbb{P}\left[\max_{0 \leq t \leq T}  V(Y_{t})>r\right]\\
&\leq \frac{2}{r-\frac{b}{1-\lambda}} \left(\sum_{t=0}^{T} \lambda^{t} r' \right)\leq \frac{2}{r-\frac{b}{1-\lambda}} \left(\frac{r'}{1 - \lambda}\right) \\
&=\frac{2r'}{d(1-\lambda)-b}\leq\frac{2r'}{b+24r'}\leq \frac{1}{12}.
\end{align*}
By the second part of the definition of the dominating chain in \cref{DeftDriftHitComp}, it is possible to couple $\{X_{t},Y_{t}\}$ so that they are equal until (at least) the first time that either leaves $C$. Applying the bound immediately above this paragraph and the assumption $C\supset \{x\in \cX\, : \, V(x)\leq r\}$, we find
\[
\mathbb{P}[X_{T} \neq Y_{T}] \leq \mathbb{P}\left[\max_{0 \leq t \leq T} \{V(X_{t}), V(Y_{t})\}>r\right] \leq \frac{1}{12}.
\]
Applying this with \cref{IneqTraceMinnow}, we have $\tvd{g(x,T,\cdot)}{\pi^{(C)}(\cdot)}\leq \frac{1}{4} + \frac{1}{12} = \frac{1}{3}$. By triangle inequality again, we have $\sup_{x,y\in C'}\tvd{g(x,t,\cdot)}{g(y,t,\cdot)}\leq \frac{2}{3}$ for all $t>T=t_{m}^{(C)}(C')$. 
\end{proof}

The following result is a minor modification of {\citep[][Thm.~12]{rosenthal1995minorization}}. Before stating it, we recall the \textit{pseudo-minorization} condition of \cite{Roberts2001}. The usual minorization condition, Inequality \eqref{EqDefMin}, requires that $g(x,t_{0},\cdot)$ be ``minorized" by a single measure $\mu$ for all $x \in C$:
\[
g(x,t_0,\cdot) \geq \epsilon \, \mu(\cdot).
\]
The pseudo-minorization condition relaxes this, requiring only that all pairs $x, y \in C$ have some minorizing measure $\mu_{x,y}$ satisfying:
\[ \label{EqPseudoMDef}
g(x,t_{0},\cdot), \, g(y,t_{0},\cdot) \geq \epsilon \mu_{x,y}(\cdot).
\]

\begin{lemma}\label{jeffmodify}
Let $g$ satisfy \cref{EqDefDrift}. 
Let $S$ be a measurable subset of $\cX$ such that 
\begin{enumerate}
\item $S\supset\{x\in \cX: V(x)\leq r\}$ for some $r>\frac{2b}{1-\lambda}$. 
\item $\sup_{x,y\in S}\| g(x,1,\cdot) - g(y,1,\cdot) \|_{\mathrm{TV}}\leq 1-\epsilon$ for some $\epsilon>0$.
\end{enumerate}
Then for every $0<p<1$ and every $x\in \cX$ we have 
\[\label{equ_temp_RHS}
\begin{split}
&\|g(x,t,\cdot)-\pi(\cdot)\|_{\mathrm{TV}}\\
&\leq (1-\epsilon)^{pt}+\left(1+\frac{b}{1-\lambda}+V(x)\right)\left[\left(\frac{1+2b+\lambda r}{1+r}\right)^{1-p}(1+2(\lambda r+b))^{p}\right]^{t}.
\end{split}
\]
\end{lemma}
	\begin{proof}
		We first show the condition $\sup_{x,y\in S}\| g(x,1,\cdot) - g(y,1,\cdot) \|_{\mathrm{TV}}\leq 1-\epsilon$ implies a pseudo-minorization condition from \cite{Roberts2001}, which is weaker than the minorization condition. By the definition of total variation distance, for any $x,y\in S$, there exists $C_{xy}\in \BorelSets \cX$ such that
		\[
		\| g(x,1,\cdot) - g(y,1,\cdot) \|_{\mathrm{TV}}= g(x,1,C_{xy}) - g(y,1,C_{xy}) \leq 1-\epsilon.
		\]
		Then we have $g(x,1,C_{xy}^c) + g(y,1,C_{xy})\ge\epsilon$. For any $A\in  \BorelSets \cX$, define
		\[
		\mu_{xy}(A):=\frac{g(y,1,A\cap C_{xy})+g(x,1,A\cap C_{xy}^c)}{g(y,1,C_{xy})+g(x,1,C_{xy}^c)}.
		\]
		Then it can be verified that $\mu_{xy}(\cdot)$ is a valid probability measure and
		\[
		g(x,1,\cdot)\ge \epsilon \mu_{xy}(\cdot),\quad g(y,1,\cdot)\ge \epsilon \mu_{xy}(\cdot).
		\]
		Therefore, $S$ is a $(1,\epsilon,\{\mu_{xy}\})$-pseudo-small set in the sense of \cite{Roberts2001}. The result then follows directly from \cite[Proposition 2]{Roberts2001}.
	\end{proof}

\begin{proof} [Proof of \cref{tracegeoergodic}]
We first apply \cref{traceminor} to obtain the second condition in \cref{jeffmodify} and then apply \cref{jeffmodify}. To see why the r.h.s.~ of \cref{equ_temp_RHS} can be written in the form of $M(x)(1-\rho)^t$, one first choose $p$ such that $(1-\epsilon)^p=\left(\frac{1+2b+\lambda r}{1+r}\right)^{1-p}(1+2(\lambda r+b))^p=:(1-\rho)$, then choose $M(x):=2+\frac{b}{1-\lambda}+V(x)$ to get the desired result.  
\end{proof}

\section{Application to Gibbs and Metropolis--Hastings Samplers} \label{SecUseGMH}

\subsection{Application to Gibbs Samplers}

We begin by studying a large class of Gibbs samplers defined in our companion paper \citep{anderson2019mixhit2}. We consider a class of Gibbs samplers targeting a measure $\pi$ supported on a compact subset of Euclidean space; without loss of generality we assume that the support of $\pi$ is inside of $[0,1]^{d}$. For $j \in \{1,2,\ldots,d\}$, let $p_{j} \, : \, \mathbb{R}^{d} \mapsto \mathbb{R}$ denote the projection to the $j$-th coordinate; that is, $p_{j}(x_{1},\ldots,x_{d}) = x_{j}$. For $x \in \cX$ and $j \in \{1,2,\ldots,d\}$, let
\[
A_{x,j} = \{a \in [0,1]^{d} \, : \, \forall \, i \neq j, \, {p_{i}(a) = p_{i}(x)}\},
\]
be the line in $[0,1]^{d}$ that passes through $x$ and is parallel to the $j$'th coordinate axis. Let $B_{x,j}$ be the connected component (the largest connected subset) of $A_{x,j}$ that contains $x$. For the remainder of this section, we consider fixed $H \in \{A,B\}$, so that \textit{e.g.}\ either $H_{x,j} = A_{x,j}$ or $H_{x,j} = B_{x,j}$. 

We note that our sets $H_{x,j}$ are always the intersection of a hyperplane with some set. When we write $\int_{H_{x,j}} f(y) \dee y$, the measure ``$\dee y$" always represents the Lebesgue measure on this hyperplane, \textit{not} Lebesgue measure on all of $\mathbb{R}^{d}$. We make the following assumption on $\pi$: 

\begin{assumption}\label{targetcts}
$\pi$  has a continuous density function $\rho$ with respect to the Lebesgue measure. Furthermore, $\int_{H_{x,j}} \rho(y) \dee y > 0$ for every $x \in \cX$ and $j \in \{1,2,\ldots,d\}$.
\end{assumption}

We now set some further notation. Set $\cX = [0,1]^{d}$ and let $\pi_{x,j}$ be the usual conditional distribution of $\pi$ on $H_{x,j}$ (that is, the measure with density $\rho$ and support $H_{x,j}$). Define the typical Gibbs sampler with target $\pi$ by
\[\label{defngibbs}
g(x,1,A) = \frac{1}{d} \sum_{j=1}^{d} \pi_{x,j}(A)
\]
for every $x\in \cX$ and $A\in \BorelSets \cX$. \textit{In the context of the Gibbs sampler only}, the $C$-dominated chain $g^{(C)}$ will \textit{always} refer to the chain from Definition \ref{DefRestGibbs}.

\begin{theorem}[{\citep[][Corollary.~4.5]{anderson2019mixhit2}}]\label{finalcor}
Let $0 < \alpha<\frac{1}{2}$. 
Suppose the target distribution $\pi$ is supported on a compact subset of Euclidean space.  
Then there exist universal constants $d_{\alpha},d'_{\alpha}$ with the following property: for every $\{g(x,1,\cdot)\}_{x\in X}$ of the form \cref{defngibbs} with $\pi$ satisfying  \cref{targetcts}, we have
\[
d'_{\alpha}t_{H}(\alpha)\leq t_{L}\leq d_{\alpha}t_{H}(\alpha).
\]
\end{theorem}

\begin{remark} \label{RemOptConst}
In the remainder of the paper, we will always take $d'_{\alpha}, d_{\alpha}$ to be the \textit{best} constants with this property; it is clear that such optimal constants exist, even if they are not obtained by the particular arguments cited above.

We note that these \textit{best} constants are monotone in $\alpha$, which will be used later.
\end{remark}

%

We now state the following elementary facts about drift conditions. 

\begin{lemma}\label{driftlemma}
Suppose $g$ is a transition kernel satisfying \cref{EqDefDrift} with function $V$ and constants $0\leq\lambda<1$ and $0\leq b<\infty$. Recall the lazy chain $g_{L}(x,1,\cdot)=\frac{1}{2} g(x,1,\cdot)+\frac{1}{2} \delta_{x}(\cdot)$. 
Then we have
\[ \label{ElFact1}
(g_{L}V)(x)=\int V(y) g_{L}(x,1,\dee y) \leq \frac{1+\lambda}{2}V(x) + \frac{b}{2}.
\]

Denoting by $g^{(C)}$ any dominated chain of the form \ref{DefResMH}
or \ref{DefRestGibbs} and $g_{L}^{(C)}$ the corresponding lazy chain, we also have
\[ \label{ElFact2}
(g^{(C)} V)(x) \leq (gV)(x), \qquad (g^{(C)}_{L} V)(x) \leq (g_{L}V)(x)
\]
for any set $C$ of the form $C = \{y \, : \, V(y) \leq r \}$, any $r > \inf_{y} V(y)$ and any $x \in C$.
\end{lemma}

\begin{proof}
To prove \cref{ElFact1}, we have:
\[
(g_{L}V)(x) = \frac{1}{2} \int V(y) g(x,1,\dee y) + \frac{1}{2} V(x) \leq  \frac{1+\lambda}{2}V(x) + \frac{b}{2}.
\]
To prove the first part of \cref{ElFact2} for chains of the form \cref{DefResMH}, we have
\[
(g^{(C)} V)(x) - (gV)(x) = \int_{y \in C} (V(x) - V(y)) g(x,1,dy) \leq g(x,1,C^{c})(V(x) - r) \leq 0.
\]
The remaining cases are similar short calculations.

\end{proof}

We now present the main result in this section:

\begin{lemma}\label{Gibbstechlemma}
There exist a universal set of constants $D = \{d_{\alpha}\}_{0 < \alpha < 0.5}$ with the following property: 
Any Gibbs sampler $g$ that satisfies Assumption \ref{targetcts} and conditions \textbf{(1)}, \textbf{(4)} of Definition \ref{DeftDriftHitComp} for some compact set $C$ of the form $C = \{x \, : \, V(x) \leq r \}$ in fact satisfies all four conditions of Definition \ref{DeftDriftHitComp} with this choice of $D$.
\end{lemma}

\begin{proof}
Condition \textbf{(2)} is immediate from \cref{driftlemma}. Condition \textbf{(3)} is exactly the content of \cref{finalcor}.

\end{proof}

Putting together \cref{Gibbstechlemma} and \cref{tracegeoergodic} immediately gives:

\begin{theorem} \label{ThmMainConsGibbs}
Let $g$ be the transition kernel of a Gibbs sampler with targeting distribution $\pi$. 
Suppose $g$ satisfies \cref{EqDefDrift} and \cref{DefEqHitCond} with $\alpha$ compatibly for some compact $C\in \BorelSets \cX$ with $\pi(C)>0$. Suppose $\pi^{(C)}$ satisfies \cref{targetcts}.
Suppose
\[
C \supset \{x \in \cX \, : \, V(x) \leq r\}
\]
for some $r' > \frac{2b}{1 -\lambda}$,  $r>\frac{2b_{1}+24r'}{1-\lambda_{1}}$ where $\lambda_{1}=\frac{1+\lambda}{2}$ and $b_1=\frac{b}{2}$.
Then there exists $M \, : \, \cX \mapsto [0,\infty)$ and $0 < \rho \leq 1$ such that 
\[
\| g_{L}(x,t,\cdot) - \pi(\cdot) \|_{\mathrm{TV}} \leq M(x) (1 - \rho)^{t}
\]
for all $x \in \cX$ and $t \in \Nats$. Moreover, $M$ and $\rho$ depend only on the parameters $\lambda, b, \alpha, C$.
\end{theorem} 

\subsection{Application to Metropolis--Hastings Chain}
In this section, we give results analogous to \cref{Gibbstechlemma} for the following class of Metropolis--Hastings chain defined in  \citep{anderson2019mixhit2}:

\begin{definition}[Metropolis--Hastings Chain]\label{defMHchain}
Fix a distribution $\pi$ with continuous density function $\rho$ supported on a subset of Euclidean space; without loss of generality we assume that the support of $\pi$ is a subset of $[0,1]^{d}$. We fix a reversible kernel $q$ with unique stationary measure $\nu$ whose support contains that of $\pi$. We assume that $\nu$ has continuous density $\phi$ and that, for all $x$ for which it is defined, $q(x,1,\cdot)$ has Lebesgue density $q_x$. Finally, we assume that the mapping $(x,y)\to q_{x}(y)$ is continuous. We define the \textit{acceptance function} by the formula
\[
\beta(x,y) = \min\left\{1, \frac{\rho(y) q_{y}(x)}{\rho(x) q_{x}(y)}\right\}.
\]
Define $g$ to be the transition kernel given by the formula
\[\label{gtranform}
g(x,1,A) = \int_{y \in A} q_{x}(y) \beta(x,y) \dee y + \delta(x,A) \int_{[0,1]^{d}} q_{x}(y) (1 - \beta(x,y))\dee y.
\]

Note that $g$ is defined uniquely by its \textit{target} $\rho$ and \textit{proposal} $q$.
\end{definition}

\begin{theorem}[{\citep[][Corollary~6.5]{anderson2019mixhit2}}]\label{mixhitmetrothm}
Let $0 < \alpha<\frac{1}{2}$. Then there exist universal constants $d_{\alpha},d'_{\alpha}$ with the following property: for every Metropolis--Hastings chain $g$ of the form \cref{gtranform} satisfying conditions in \cref{defMHchain} such that its target distribution has continuous density function with compact support, we have
\[
d'_{\alpha}t_{H}(\alpha)\leq t_{L}\leq d_{\alpha}t_{H}(\alpha).
\]
\end{theorem}

\textit{In the context of the Metropolis--Hastings sampler only}, the $C$-dominated chain $g^{(C)}$ will \textit{always} refer to the chain from Definition \ref{DefResMH}. We note that \cref{driftlemma} remains true with ``Gibbs" replaced by ``Metropolis--Hastings" the one time it is used. As discussed in Remark \ref{RemOptConst}, we always use the \textit{optimal} constants when this theorem is invoked, and these optimal constants are monotone in $\alpha$.

We then have the main result of this section:

\begin{lemma}\label{mixhitmetro}
There exist a universal set of constants $D = \{d_{\alpha}\}_{0 < \alpha < 0.5}$ with the following property:

Any Metropolis--Hastings sampler $g$ of the form \cref{gtranform} satisfying conditions in \cref{defMHchain} and conditions \textbf{(1)}, \textbf{(4)} of Definition \ref{DeftDriftHitComp} for some compact set $C$ of the form $C = \{x \, : \, V(x) \leq r \}$ in fact satisfies all four conditions of Definition \ref{DeftDriftHitComp} with this choice of $D$.
\end{lemma}

\begin{proof}
Condition \textbf{(2)} is immediate from \cref{driftlemma}\footnote{As we have just noted, this applies with ``Gibbs'' replaced by ``Metropolis--Hastings''.}. Condition \textbf{(3)} is exactly the content of \cref{mixhitmetrothm}.

\end{proof}

Combining \cref{tracegeoergodic} and \cref{mixhitmetro} tells us that \cref{ThmMainConsGibbs} holds for Metropolis--Hastings samplers as well as Gibbs samplers (we don't rewrite the theorem to save space).

\section{Comparison to Minorization and Illustrative Example} \label{SecAppl}

The main goal of this paper was to illustrate how the minorization condition \cref{IneqDefMin} in \cref{ThmDriftMinClassic} can be replaced by a hitting-time condition (note that, in both cases, a drift condition is still required to conclude that a chain is geometrically ergodic). We have not yet seriously addressed the obvious question: why would anyone wish to do this?

We would like to be able to say that using maximum hitting times gives better bounds. However, the  equivalence results of \cite{oliveira2012mixing,finitemixhit,basu2015characterization, anderson2018mixhit,anderson2019mixhit2} say that this nice response cannot be correct. After all, if mixing times and maximum hitting times are equal up to some universal constants, methods based on maximum hitting times must give bounds that are essentially equivalent to those obtained from the pseudo-minorization condition of \cite{Roberts2001}.

Instead of obtaining \textit{better} results, we claim that bounds on the maximum hitting time often have \textit{easier proofs.} The ``difficulty" of a proof is of course subjective, but we give evidence based on the following concrete claims:

\begin{enumerate}
\item \textbf{Maximum hitting time bounds are \textit{never} ``harder" than pseudo-minorization bounds:} Any pseudo-minorization bound implies a maximum hitting time bound via a straightforward and very short proof (see the argument leading up to \cref{IneqMixHitTrivialDirection}). 
\item \textbf{In some realistic examples, we can compute maximum hitting bounds but not pseudo-minorization bounds:} We have a broad class of simple but realistic examples for which we can compute good bounds on the maximum hitting time, but have no realistic way to directly compute good bounds on the mixing time (see \cref{ThmSimpleApp}). 
\end{enumerate}

Before illustrating the first, we say that a kernel $g$ satisfies the $t$-step pseudo-minorization condition of \cite{Roberts2001} on a set $S$ with constant $\epsilon > 0$ if 
\[
\tvd{g(x,t,\cdot)}{g(y,t,\cdot)} \leq (1-\epsilon)
\]
for all $x,y \in S$. We note that any chain with a mixing time satisfies this pseudo-minorization condition with $t$ equal to the mixing time, $S$ equal to the whole state space and $\epsilon = \frac{1}{4}$. On the other hand, any chain satisfying a pseudo-minorization condition on the whole state has a finite mixing time. 

We now illustrate the first point. Consider \textit{any} Markov chain with transition kernel $g$,  unique stationary measure $\pi$ and mixing time $t_{m}$. Then for any measurable $A$ with $\pi(A) \geq \frac{1}{3}$ and starting point $x$,
\[
g(x,t_{m},A) \geq \pi(A) - \| g(x,t_{m},\cdot) - \pi(\cdot) \|_{\mathrm{TV}} \geq \frac{1}{12}.
\]
This implies, for all starting points $x$, measurable sets $A$ with $\pi(A) \geq \frac{1}{3}$, and integers $k \in \mathbb{N}$,
\[
\mathbb{P}_{x}\left[\tau(A) > k t_{m} \, | \, \tau(A) > (k-1) t_{m}\right] \leq \frac{11}{12}.
\] 
Using the integration-by-parts formula, this implies 
\[
\mathbb{E}_{x}[\tau(A)] \leq t_{m} \, \sum_{k=0}^{\infty} \mathbb{P}_{x}[\tau(A) > k t_{m}] \leq t_{m} \sum_{k=0}^{\infty} \left(1-\frac{1}{12}\right)^{k} = 12 t_{m},
\]
and so
\[ \label{IneqMixHitTrivialDirection}
t_{H}\left(\frac{1}{3} \right) \leq 12 t_{m}.
\]
This completes our short, elementary and self-contained argument that mixing conditions imply hitting conditions.

Next, we consider a class of examples with state space $[0,1]$. We begin by defining the following class of nearly-unimodal measures. These are meant to mimic the typical ``small set" used in a pseudo-minorization argument: we expect the target distribution to be reasonably close to unimodal on the small set, but in practice we often do not have detailed control over the size of the small set or the fluctuations of the target distribution over the small set.

\begin{definition}
Fix $0 < \alpha < 0.5$ and $1 \leq r < \infty$. For a distribution or density $\rho$ and constant $0< q< 1$, denote by $m(\rho,q)$ the $q$'th quantile of $\rho$. We say that a density $\rho$ on $[0,1]$ is \textit{unimodal} if there exists some $m\in [0,1]$ such that
\[
\rho(x) < \rho(y)
\]
for all $0 \leq x < y \leq m$ and 
\[
\rho(x) > \rho(y)
\]
for all $m \leq x < y \leq 1$. We say that a density $\psi$ is $(\alpha,r,\rho)$-nearly unimodal if
\begin{enumerate}
\item $\rho$ is unimodal, \textit{and}
\item $r^{-1} \leq \frac{\rho(x)}{\psi(x)} \leq r$ for all $x \in [0,1]$, \textit{and}
\item $r^{-1} \leq \frac{\psi(x)}{\psi(y)} \leq r$ for all $x,y \in [m(\psi,\alpha), m(\psi,1-\alpha)]$.
\end{enumerate}
\end{definition}

We have the following:

\begin{theorem} \label{ThmSimpleApp}
Fix $1 \leq r < \infty$, $0 < \alpha < 0.5$. There exists a universal constant $L = L(\alpha,r)$ so that the following holds:

Let $\psi$ be a $C^{1}$ density that is $(\alpha,r,\rho)$-nearly unimodal for some $\rho$. For $c^{-1} \in \{1,2,\ldots\}$, let $g_{c}$ be the transition kernel given in  \cref{defMHchain} with proposal kernel $q_{c}(x,1,A) = \frac{|A \cap[x-c,x+c]|}{2c}$ and target $\psi$. Then the maximum hitting time $t_{H}(g_{c},\alpha)$ of $g_{c}$ satisfies 
\[
t_{H}(g_{c},\alpha) \leq L c^{-2}
\]
for all $c^{-1} > C(\psi,\alpha)$ sufficiently large.
\end{theorem}
 
\begin{proof}

The basic strategy is to define a simple sequence of discretizations of the process, bound their maximal hitting times, then use weak convergence results to translate this bound back to our original kernel $g_{c}$. We take this slightly indirect route in order to take advantage of some particularly-simple formulas for expected hitting times that are available in the discrete setting.

For $c^{-1} \in \mathbb{N}$, we define the transition kernel of a closely-related ``birth-and-death" chain $h_{c}$ on $[[c]] \equiv \{0,c,2c,\ldots,1-2c,1-c\}$ by the formula
\begin{align*}
h_{c}(x,1,x+1) &= \frac{1}{2} \min \left\{1, \frac{\psi(x+c)}{\psi(x)} \right\}, \qquad x < 1 \\
h_{c}(x,1,x-1) &= \frac{1}{2} \min \left\{1, \frac{\psi(x-c)}{\psi(x)} \right\}, \qquad x >0 \\
\end{align*}
and $h_{c}(x,1,x) = 1 - h_{c}(x,1,x-1) - h_{c}(x,1,x+1)$. Since $h_{c}$ is a birth-and-death chain, it is reversible with respect to some probability distribution function $\psi_{c}$. 

Next, fix a measurable set $S$ with $\int_{x \in S} \psi(x) dx > \alpha$. By the (uniform) continuity of $\psi$, it is straightforward to check that $\sum_{x \in S \cap [[c]]} \psi_{c}(x) >  \frac{9}{10} \alpha$ for all $c > C(\psi, \alpha)$ sufficiently large. Since $\int_{x \in S} \psi(x) dx > \alpha$ and the entire state space of the Markov chain is the interval $[0,1]$, for $0 < \alpha' < \frac{1}{2} \alpha$ there exists a point $p(c,\alpha') \in \{ c \lfloor c^{-1} m(\psi,1-\alpha') \rfloor, \ldots, c \lceil c^{-1} m(\psi,1-\alpha') \rceil \}$ in the ``middle" of the interval $[[c]]$ around which $S$ is fairly dense:
\[ \label{IneqHighDensityPoint}
\int_{x \in S \cap [p(c,\alpha'), p(c,\alpha') + c]} 1 dx > (\alpha - 2 \alpha') c >0.
\]
For convenience, define $m'(\psi,\alpha') =  c \lfloor c^{-1} m(\psi,1-\alpha') \rfloor$ and $m'(\psi,1-\alpha') = c \lceil c^{-1} m(\psi,1-\alpha') \rceil$. Note that since $\alpha' < 0.5$ and $(1-\alpha') > 0.5$, there is no ambiguity in the notation for $m'$.

Denote by $\tau(h,A)$ the hitting time of a measurable set $A$ for a transition kernel $h$, extending the notion in Equation \cref{EqBasicHitDef}. Since birth-and-death chains cannot ``skip" any values,
\begin{equation}
\label{IneqNoSkip}
\begin{split}
&\max_{x \in [[c]]} \mathbb{E}_{x}[\tau(h_{c},\{p(c,\alpha')\})]\\
&\leq \max (\mathbb{E}_{0}[\tau(h_{c},\{  m(\psi_{c}, 1-\alpha')\})],\mathbb{E}_{1-c}[\tau(h_{c},\{  m(\psi_{c},\alpha')\})]) = t_{H}(h_{c},\alpha').
\end{split}
\end{equation}

Crucially, we have bounded the left-hand side by expressions that do not depend on $S$.

The expectations appearing in the middle expression of \cref{IneqNoSkip} have well-known (though slightly messy) explicit formulas. We now inspect the formula appearing in  \cite[Theorem 2.3]{palacios1996note}. Comparing the explicit formula for $h_{c}$ to the explicit formula for the $\frac{1}{2}$-lazy simple random walk $w_{c}$ on the interval $\{0,c,2c,\ldots,1-2c,1-c\}$, we see that every term in the explicit formula differs between these two walks by at most a multiplicative factor depending only on $\alpha, r$. Thus, there exists a universal constant $L_{1} = L_{1}(\alpha',r)$ such that   
\[ \label{HWRelation}
 t_{H}(h_{c},\alpha') \leq L_{1}  t_{H}(w_{c},\alpha'). 
\]
Since $t_{H}(w_{c},\alpha') \leq c^{-2}$ for all $\alpha' > 0$ (see \textit{e.g.}  \cite[Example 10.20]{markovmix}), \cref{IneqNoSkip} and \cref{HWRelation} together give
\[ \label{IneqHittingBDBD}
\max_{x \in [[c]]} \mathbb{E}_{x}[\tau(h_{c}, p(c,\alpha')\})] \leq L_{1} c^{-2}.
\]
Note that the bound on the right-hand side is uniform over the choice of set $S$.

By \cref{IneqHighDensityPoint}, there exists a universal constant $L_{2} = L_{2}(\alpha',r)$ such that
\[ \label{IneqGcInterval}
\sup_{x} \mathbb{E}_{x}[\tau(g_{c},S)] \leq L_{2} \, \sup_{x} \max(\mathbb{E}_{x}[\tau(g_{c},[0, m'(\psi,\alpha')])],\mathbb{E}_{x}[\tau(g_{c},[m'(\psi,1-\alpha'),1])] ].
\]

We now relate this to $h_{c}$.  Fix some $B > 0$ and consider sample paths $\{ X_{t}^{(c)}\}_{t=0}^{B c^{-2}} \sim g_{c}$, $\{Y_{t}^{(c)}\}_{t=0}^{B c^{-2}} \sim h_{c}$. After rescaling time by a constant factor, the stochastic processes  $\{ X_{c^{-2}t}^{(c)}\}_{t=0}^{B}, \, \{Y_{c^{-2}t}^{(c)}\}_{t=0}^{B}$ converge to the same limiting continuous-time process in Skorohod's topology on $\mathcal{D}[0,B]$ as $c \rightarrow 0$; the proof of this fact is deferred to the appendix, where it is described as Lemma \ref{LemmaSkorohodArg}.

 Since $g_{c}$ cannot jump more than distance $c$, this convergence (together with Lemma \ref{NewSkoroLemma}) implies  that for all $\epsilon > 0$ there is a constant $L_{3} = L_{3}(\epsilon,\alpha,\alpha',r)$ such that\footnote{We note that Lemma \ref{NewSkoroLemma} relates hitting times of two slightly different sets, and so at first glance this might cause a problem. However, there is no problem: our bounds apply to \textit{all} choices of $\alpha'$ in an open interval, and are monotone in the choice of $\alpha'$ within that interval. Thus, Lemma \ref{NewSkoroLemma} lets us compare hitting times at the cost of an arbitrarily-small change of our choice of $\alpha'$ within that interval and changing the implied constant.}
\[
\sup_{x} \mathbb{E}_{x}[\tau(g_{c},I)] \leq L_{3} \left(\max_{x \in [[c]]} \max_{I \in \mathcal{I}(\alpha)}\mathbb{E}_{x}[\tau(h_{c},I)] + \epsilon c^{-2}\right),
\]
where $\mathcal{I}(\alpha') =\{ [0, m'(\psi,\alpha')], [m'(\psi,1-\alpha'),1]\}$. Combining this with Inequalities \cref{IneqHittingBDBD} and \cref{IneqGcInterval}, we have

\[
\sup_{x} \mathbb{E}_{x}[\tau(g_{c},S)] \leq L_{2} L_{3} \, \left(\max_{x \in [[c]]} \max_{I \in \mathcal{I}(\alpha)}\mathbb{E}_{x}[\tau(h_{c},I)] + \epsilon c^{-2}\right) = O( c^{-2}).
\]
Since these bounds all hold uniformly over the choice of measurable $S$ with stationary measure greater than $\alpha$, we can fix any $\alpha' \in (0, \frac{1}{2} \alpha) \neq \emptyset$ and this bound completes the proof.
\end{proof}

When $\psi(x) \equiv 1$, the kernel analyzed above is exactly the usual ``ball" walk on $[0,1]$, so this upper bound is in fact sharp. While our analysis is fairly simple, we know of no simple way to get comparable results by directly analyzing the mixing time. See further discussion of this question in \textit{e.g.} \cite{johndrow2018fast,yuen2000applications}, where very similar classes of Markov chains are studied using other methods; as shown in those articles, these other methods cannot give bounds better than $t_{m}(g_{c}) = O(c^{-3})$ for any walk in this class.

As in those other articles, we are often interested in Markov chains that are supported on all of $\mathbb{R}$ rather than merely a compact set. These chains will typically not have finite expected hitting times, and Theorem \ref{ThmSimpleApp} will not yield any nontrivial bounds. In this setting, Theorem \ref{ThmSimpleApp} can be combined with our main result, Theorem \ref{tracegeoergodic}, to obtain useful convergence estimates.
\printbibliography

\appendix 

\section{Short Results on Convergence of Hitting Times and Skorohod Topology}

We consider two ergodic Markov processes $\{X_{t}^{(n)}\}_{t=0}^{B}, \{Y_{t}^{(n)}\}_{t =0}^{B}$ on $[0,1]$ that converge in distribution to a common process $\{Z_{t}\}_{t=0}^{B}$ on $[0,1]$ in the Skorohod topology $\mathcal{D}[0,B]$. 

We assume that these processes all have stationary measures. Furthermore, we assume that for every measurable $A \subset [0,1]$ with strictly positive Lebesgue measure, these stationary measures assign strictly positive probability to $A$ (for $n > N(A)$ sufficiently large).

For $W \in \{X^{(n)}, Y^{(n)}, Z\}$ and $p \in [0,1]$, define 
\[
\tau(W,p) = \inf \{t \geq 0 \, : \, W_{t} \geq p\}.
\] 

The first result we need is:

\begin{lemma} \label{NewSkoroLemma}
For all $\epsilon > 0$ and $q < p \in [0,1]$,
\[
\lim_{n \rightarrow \infty} \mathbb{P}_{q}[\tau(X^{(n)}, p) >  \tau(Y^{(n)},p+\epsilon) + \epsilon] = 0.
\]

\end{lemma}

\begin{proof}

Denote by $D_{B}$ the Skorohod distance on paths on $[0,B]$. We note that, for all $\epsilon > 0$ and all pairs of paths $\{U_{t}, V_{t}\}_{t=0}^{B}$, the event
\[
\{ \tau(U,p) \leq \epsilon + \tau(V,p+\epsilon)\} \cup \{D_{B}(\{U_{t}\}_{t=0}^{B}, \{V_{t}\}_{t=0}^{B}) > \epsilon\}
\]
holds. Note that this holds pathwise, and is not probabilistic.

This immediately implies that, for all $\epsilon > 0$ and $p \in [0,1]$,
\[
\lim_{n \rightarrow \infty} \mathbb{P}[\tau(X^{(n)}, p) > \epsilon + \tau(Z,p+\epsilon)] = 0.
\]
Applying the same argument again to compare $Y^{(n)}, Z$ (and rescaling $\epsilon$) gives the desired conclusion
\[
\lim_{n \rightarrow \infty} \mathbb{P}[\tau(X^{(n)}, p) > \epsilon + \tau(Y^{(n)},p+\epsilon)] = 0.
\]

\end{proof}

We note that by \textit{e.g.} switching signs and/or the role of $X,Y$ in the theorem, we get analogous results for other hitting times such as 
\[
\tau'(W,p) = \inf \{t \geq 0 \, : \, W_{t} \leq p\}.
\]

The next result we need is:

\begin{lemma} \label{LemmaSkorohodArg}
Fix $0 < B < \infty$ and let $\{ X_{c^{-2}t}^{(c)}\}_{t=0}^{B}, \, \{Y_{c^{-2}t}^{(c)}\}_{t=0}^{B}$ be stochastic processes sampled from the generators $g_{c}, h_{c}$ appearing in the proof of Theorem \ref{ThmSimpleApp}. 

Then there exist constants $C_{X}, C_{Y}$ and a process $\{Z_{t}\}_{t=0}^{B}$ so that $\{ X_{C_{X} c^{-2}t}^{(c)}\}_{t=0}^{B}, \, \{Y_{C_{Y} c^{-2}t}^{(c)}\}_{t=0}^{B}$ both converge to  $\{Z_{t}\}_{t=0}^{B}$ in Skorohod's topology on $\mathcal{D}[0,B]$ as $c \rightarrow 0$.
\end{lemma}

\begin{proof}

Although we expect that this fact is well-known, we were not able to find a precise reference in the literature. Since weak convergence arguments can be rather long and in this case there are not interesting technical details, we instead give a short sketch.

In \cite{burdzy2008discrete}, the authors provide a general technique for proving weak convergence of certain discrete-time random walks to a reflected Brownian motion. Our walks will converge to a reflected Brownian motion \textit{with drift} and so are not directly covered by the theorems as stated. However, the proof of \cite[Thm 2.4]{burdzy2008discrete} can be modified to prove our assertion with the following changes to the lemmas used: 

\begin{enumerate}
\item \cite[Lemma 2.1]{burdzy2008discrete}: no substantial changes required - the sequence of walks is clearly tight.
\item \cite[Lemma 2.2]{burdzy2008discrete}: recall that the density of usual Brownian motion satisfies the heat equation, $u_{t} = \frac{1}{2} \Delta u$. For Brownian motion with drift and stationary measure $\rho$, the density satisfies an analogous diffusion $u_{t} = \nabla \rho(u) + \frac{1}{2} \Delta u \equiv \Delta_{\rho} u$ containing a gradient term. In \cite[Lemma 2.2]{burdzy2008discrete}, $\Delta$ should be replaced by $\Delta_{\rho}$.  
\item \cite[Lemma 2.3]{burdzy2008discrete}: no changes required at all.
\end{enumerate}

Propagating the above changes through the proof of \cite[Thm 2.4]{burdzy2008discrete} gives the desired conclusion. We note that the proof of \cite[Thm 2.4]{burdzy2008discrete} is fairly long because \cite{burdzy2008discrete} is concerned with a much harder problem: proving convergence to reflected random walk in higher dimensions and with non-smooth boundaries. Proving convergence in one dimension, as we do here, is substantially simpler.

We give here also a sketch of an elementary proof that mimics the strategy of \cite{burdzy2008discrete}, for readers who do not wish to look at the full details of that paper.

We begin by setting notation. Recall that a sequence of random variables $V_{n}$ converges weakly to a random variable $V$ if $\mathbb{E}[f(V_{n})] \rightarrow f(V)$ for all bounded, continuous functions $f$. We say that this weak convergence \textit{occurs up to additive error $e_{n}$ with constant $C < \infty$} if 
\[ 
|\mathbb{E}[f(V_{n})] - f(V)| \leq C e_{n}
\] 
for all functions $f$ with $\|f \|_{\infty} = 1$. Recall also the definition of the \textit{trace} of a CADLAG stochastic process $X_{t}$ on a set $S$. Roughly speaking, the trace is obtained by throwing out all points $t$ for which $X_{t} \notin S$; see \cite{landim2019metastable} for a precise definition.

Continuing, note that we have explicit formulas for the generators of the stochastic processes of interest. For fixed $\gamma >0$ and all  $0 < \delta < \Delta(\gamma) >0$ sufficiently small, Donsker's theorem tells us that  $\{ X_{C_{X} c^{-2}t}^{(c)}\}_{t=0}^{\delta}, \, \{Y_{C_{Y}c^{-2}t}^{(c)}\}_{t=0}^{\delta}$ started at common point $X_{0} = Y_{0} = x \in [\gamma, 1-\gamma]$ both converge weakly to a Brownian motion with drift $\psi'(x)$, and furthermore that this convergence occurs up to additive error $\delta^{2}$ with some constant $C = C(\gamma)$ that is uniform in the choice of $x \in [\gamma, 1-\gamma]$ and $0 < \delta < \Delta(\gamma)$. 

For fixed $0 < \gamma <0.25$, let $\{\hat{X}_{t}^{(c)}\}$, $\{\hat{Y}_{t}^{(c)}\}$ be the \textit{traces} of $\{X_{t}^{(c)}\}$, $\{Y_{t}^{(c)}\}$ on $[\gamma,1-\gamma]$. Concatenating paths, the above calculation shows that  $\{ \hat{X}_{C_{X} c^{-2}t}^{(c)}\}_{t=0}^{B}, \, \{\hat{Y}_{C_{Y}c^{-2}t}^{(c)}\}_{t=0}^{B}$ converge to the \textit{trace} of Brownian motion with drift on $[\gamma,1-\gamma]$, with additive error $B \delta$ and the same constant $C = C(\gamma)$ as above. 

Next, it is possible to check that, for any fixed $t > 0$ and $Z \in \{X,Y\}$, $P[Z_{C_{Z} c^{-2}t}^{(c)} \in [0,\gamma) \cup (1-\gamma,1]]$ goes to 0 uniformly in the starting point of the process (though not uniformly in $t$).\footnote{The bound on the time that an excursion into $[0,\gamma)$ lasts can be proved \textit{e.g.} by coupling to a pair of random walks with constant drift just above or just below $\psi'(0)$ in such a way that our walk is always sandwiched between the constant-drift walks until the first exit time from $[0,\gamma]$. The excursion times for the constant-drift walks themselves can be directly bounded using \textit{e.g.} Hoeffding's inequality.} Thus, as $\gamma$ goes to 0, the trace processes themselves converge to some limiting process\footnote{This follows immediately from \textit{e.g.} Theorem 2.5 of \cite{ethier2009markov}}. Since the traces converge as $c$ goes to 0 for any fixed $\gamma > 0$, and the traces converge as $\gamma$ goes to 0 (and the occupation time of $[0,\gamma]\cup[1-\gamma,1]$ converges weakly to 0), the full processes must converge as $c$ goes to 0 as well.

Since the limits of the trace processes are themselves continuous, the generator of this limiting process must agree with the generator of Brownian motion with drift on every interval of the form $[\gamma,1-\gamma]$; by the same ``sandwich" argument described above, it is furthermore possible to check that the limiting process cannot have any jumps at $\{0\}$ or $\{1\}$. The only generator with both these properties is reflected Brownian motion with drift.

\end{proof}

\end{document}